\documentclass[12pt,reqno]{amsart}
\usepackage{latexsym, amsmath, amscd, amssymb, amsthm,longtable} 
\usepackage[T2A]{fontenc}
 \usepackage[cp1251]{inputenc}
\usepackage[english]{babel}
\newtheorem{theorem}{Theorem}[section]

\usepackage{color}

\begin{document}

\noindent

\title[Moving Frames and  Graph Canonization  ]{Discrete Moving Frames, Semi-Algebraic Invariants  and the Graph Canonization Problem}

\author{Leonid Bedratyuk}

\begin{abstract}
This paper develops an invariant--geometric interpretation of the canonization problem for simple undirected weighted graphs based on the {discrete moving frame method} for finite groups. We consider the action of the {pair group} $S_n^{(2)}$ on the space of edge weights of a graph. It is emphasized that the classical algebraic approach aimed at describing the ring of polynomial invariants of this action quickly becomes computationally impractical due to the explosive growth in the number and degrees of generators.
The main result is a formalization of a canonical labeling of a graph as a {discrete moving frame} in the sense of Olver: a discrete orbit cross-section is fixed, in particular by a lexicographic rule, and for each configuration of edge weights one defines a permutation in $S_n^{(2)}$ that maps it to its canonical representative. The coordinates of the canonical representative are interpreted as a {complete system of invariants} for the action of $S_n^{(2)}$ that separates orbits, i.e., isomorphism classes of graphs.
It is shown that the invariants obtained via such orbit canonization are of a non-algebraic nature and belong to the class of {semi-algebraic functions}.
We do not propose a new computational algorithm; instead, we provide a rigorous theoretical foundation for the very concept of canonization by viewing it as a process of constructing a discrete moving frame and the corresponding system of semi-algebraic invariants.

\vspace{0.5cm}
\end{abstract}

\maketitle

\section{Introduction}

The graph canonization problem is one of the central problems in graph theory and algorithmic combinatorics. Its goal is to construct a canonical representative of the isomorphism class of a graph---an object that is invariant for all graphs isomorphic to each other and unique for each class of non-isomorphic graphs \cite{Baba1, Baba2, SW}. Canonization is closely related to the graph isomorphism problem: an efficient canonization algorithm automatically solves the isomorphism problem, and conversely \cite{Gur, KV}. In practice, canonization is implemented via sophisticated combinatorial procedures, in particular via individualization and partition refinement, as in the classical \texttt{nauty} and \texttt{Traces} algorithms, which are currently standard tools in this area \cite{McKayPiperno2014, CodishEtAl2016}. Despite substantial applied success, the structural and conceptual interpretation of canonization remains a subject of active theoretical research.

Attempts to construct a complete graph invariant by the classical algebraic route---via describing the ring of polynomial invariants of the action of the pair group $S_n^{(2)}$ on the edge variables---have proved computationally inefficient. The rapid growth in the number and degrees of such polynomials leads to a combinatorial explosion, which makes this approach practically inapplicable even for graphs with a small number of vertices. In fact, complete systems of algebraic invariants have been explicitly described only for graphs with at most five vertices \cite{NT, Bbed}. In contrast, the use of \emph{semi-algebraic functions} (see \cite{BC}) for expressing invariants has turned out to be much more promising and has made it possible in this paper to construct a complete system of invariants for the action of the group $S_n^{(2)}$. This approach is conceptually natural, since it directly appeals to the procedure of orbit canonization, thereby establishing a fundamental link between invariant theory and algorithmic practice.

The transition from classical polynomial invariants to the broader class of semi-algebraic functions is dictated not only by the computational complexity of the ring of polynomial invariants of the group $S_n^{(2)}$, but also by fundamental limitations of purely polynomial descriptions in orbit separation problems. As shown in~\cite{Fijalkow2019}, for many group actions purely algebraic invariants are insufficient for a complete classification of configurations, whereas the semi-algebraic apparatus makes it possible to build complete systems of invariants. This thesis is further supported by the results of~\cite{Liu2011}, which indicate that using only polynomial equalities often does not yield a sufficiently tight description of invariant sets, forcing one to employ inequalities in order to model complex states precisely. From a practical perspective, as demonstrated in~\cite{Bagnara2005}, it is precisely semi-algebraic invariants that make it possible to capture properties of systems governed by logical guards, which in our context corresponds to the lexicographic comparison procedure. Thus, algorithmic selection using semi-algebraic functions makes it possible to overcome the ``algebraic dead end'' and to formulate a constructive isomorphism criterion even in cases where classical methods do not provide effective orbit separation.

The aim of this work is to propose a unified interpretation of the canonization problem for simple undirected weighted graphs in terms of {semi-algebraic} invariants of the action of the pair-symmetric group
$$
S_n^{(2)}: \quad \boldsymbol{e}_{ij} \longmapsto \boldsymbol{e}_{\pi(i)\pi(j)}, \qquad \pi \in S_n, \ i < j,
$$
on the space of edge weights $\boldsymbol{e}_{ij}$ of a graph on $n$ vertices. The key idea is that graph canonization can be viewed as the construction of a \emph{discrete  moving frame} in the sense of Olver's theory for discrete group actions \cite{Olver2023}. To this end, one fixes a discrete orbit cross-section, for example via the lexicographic minimum, and for each configuration $\boldsymbol x$ one chooses a permutation $\rho(\boldsymbol x) \in S_n^{(2)}$ that maps $\boldsymbol x$ into this cross-section.
Under this approach, the invariants become the coordinates of the canonical representative of the orbit, which corresponds to the canonical vector of edge weights or the canonical adjacency matrix.
The resulting invariant functions are {algorithmic} in nature and have a piecewise-analytic structure, which is a characteristic feature of {semi-algebraic functions}. They are defined via minimization and comparison operations over the finite set of orbit images, and therefore typically do not belong to the classical ring of polynomial invariants. In this sense, canonization generates a fundamental class of semi-algebraic invariants that form a complete graph invariant: two edge-weight configurations belong to the same orbit of the action of $S_n^{(2)}$ if and only if their invariantized coordinates coincide completely.

The structure of the paper is as follows.
In Section~1, for completeness, we recall the classical algebraic approach to graph invariants, describe standard computational constructions, and briefly summarize known results that delineate the possibilities and limitations of this approach.
In Section~2, we present the basics of the discrete moving frame method for finite groups: we define a moving frame, orbit normalization, and invariantization, and we prove that the invariantized coordinates of the canonical representative form a complete system of $G$-invariants.
In Section~3, we apply this apparatus to the action of the pair group $S_n^{(2)}$ on the space of edge weights: we construct a canonical orbit cross-section via the lexicographic minimum, define the corresponding moving frame $\rho(\boldsymbol{x})$, and obtain a complete system of invariants as the coordinates of $\rho(\boldsymbol{x})\cdot\boldsymbol{x}$. We also establish the semi-algebraic nature of the constructed invariants and provide a detailed example for $n=4$.
In the Conclusions, we summarize the results and discuss an interpretation of classical canonical labeling algorithms as efficient implementations of a discrete moving frame.

\section{The Classical Algebraic Approach to Polynomial Graph Invariants}

In this section, for completeness, we recall the classical algebraic approach to graph invariants. Within this framework, invariants are understood as polynomial functions of the edge coordinates that remain unchanged under relabeling of vertices. We give the definition of the algebra of polynomial invariants of a graph, describe standard computational constructions, and briefly summarize known results that delineate the possibilities and limitations of this approach.

\subsection{The group $S_n^{(2)}$ and its action on a graph}

Consider the set of two-element subsets
$$
E_n=\bigl\{\{i,j\}\subset[n],\quad  \bigm|\ 1\le i<j\le n\bigr\},\quad [n]=\{1,2,\dots,n\},  \quad  n>2,
\qquad |E_n|=\binom{n}{2}.
$$
Each permutation $\pi\in S_n$ of the set $[n]$ induces a permutation of the set $E_n$ by the rule
$$
\{i,j\}\longmapsto \{\pi(i),\pi(j)\}.
$$
This gives a homomorphism
$$
\varphi:S_n\longrightarrow S_{E_n}\cong S_{\binom{n}{2}},
$$
whose image is called the \textit{pair group} $S_n^{(2)}$:
$$
S_n^{(2)}=\varphi(S_n)\subset S_{\binom{n}{2}}.
$$
The groups $S_n^{(2)}$ and $S_n$ are isomorphic as abstract groups, but they are realized differently:
$S_n$ acts on vertices, whereas $S_n^{(2)}$ acts on edges. Intuitively, $S_n^{(2)}$ is the group of all permutations of edges that arise
solely from a relabeling of vertices; see also \cite[Chapter 4]{H-P}.

Denote by $\mathcal{V}_n$ the vector space of all undirected graphs on the vertex set
$[n]$ whose edges are weighted by real numbers.
For each edge $\{i,j\}\in[n]^{(2)}$ denote by $\boldsymbol e_{ij}$
the simple graph having the single edge $\{i,j\}$ of weight $1$, and let $\omega_{i,j}\boldsymbol{e}_{\{i,j\}}$ be the graph with the single edge $\{i,j\}$ of weight $\omega_{ij} \in \mathbb{R}$.
Then $\mathcal{V}_n$ is a vector space of dimension $m=\binom{n}{2}$ with basis
$$
\langle \boldsymbol e_{1\,2},\boldsymbol e_{1\,3},\ldots,\boldsymbol e_{n-1\,n}\rangle.
$$

Let now $\mathcal{V}_n^*$ be the dual space of $\mathcal{V}_n$. Its basis consists of the linear functions
$x_{ij}:\mathcal{V}_n\to\mathbb{R}$ that select the weight of the corresponding edge:
$$
x_{ij}\bigl(\boldsymbol e_{\{k,l\}}\bigr)=\delta_{ik}\delta_{jl},
\qquad 1\le i<j\le n,\ 1\le k<l\le n.
$$

We define the actions of the groups $S_n$ and $S_n^{(2)}$ on the edge coordinates.
A vertex relabeling $\sigma\in S_n$ acts on the basis edges by
$$
\sigma\cdot \boldsymbol e_{\{i,j\}}=\boldsymbol e_{\{\sigma(i),\sigma(j)\}}.
$$
The induced action on the coordinate functions (with ordered indices) is given in the standard way by
$$
\sigma \cdot x_{i\,j} = x_{\sigma^{-1}(i)\,\sigma^{-1}(j)}.
$$
This action, when restricted to an appropriately ordered set of coordinates $\{x_{i\,j}\}$, coincides with the action of the pair group
$S_n^{(2)}$ as a subgroup of $S_m$, since it permutes
the coordinates consistently with the relabeling of vertices.

\subsection{Polynomial invariants of the group $S_n^{(2)}$}

Let us briefly explain the classical approach to invariants of the group $S_n^{(2)}$.
For convenience, we introduce a new set of variables:
$$
\{x_1,x_2,\ldots,x_m \}=\{ {x}_{1\,2}, {x}_{1\,3}, \ldots, {x}_{n-1\,n}\}.
$$
Then the action of $S_n$ on the set $\{{x}_{1\,2}, {x}_{1\,3}, \ldots, {x}_{n-1\,n}\}$ induces the action of the pair group $S^{(2)}_n$ on the set $\{x_1,x_2,\ldots,x_m \}$.

We extend the action of the group $S_n^{(2)}$ to the algebra of polynomial functions
$\mathbb{K}[\mathcal{V}_n]=\mathbb{K}[x_1, x_2, \ldots, x_m]$.
Polynomial functions that remain unchanged under this action are called invariants of the pair group and form the invariant algebra
$\mathbb{K}[\mathcal{V}_n]^{S_n^{(2)}}$, which is called the invariant algebra of the vector space of simple undirected weighted graphs on $n$ vertices.
To compute invariants one uses the classical Reynolds operator, averaging over the orbit,
$$
R(f)=\frac{1}{|S_n^{(2)}|}\sum_{g\in S_n^{(2)}} (g\cdot f), \quad  f \in \mathbb{R}[\mathcal{V}_n],
$$
which is a projector $\mathbb{R}[\mathcal{V}_n] \to \mathbb{R}[\mathcal{V}_n]^{S_n^{(2)}}$.

\medskip
Consider the example of a simple graph on four vertices.
Then $m=6$, and we have the following relabeling of variables
$$
\{ x_1,\ldots,x_6\}=(x_{12},x_{13},x_{14},x_{23},x_{24},x_{34}).
$$
The invariant algebra $\mathbb{K}[\mathcal{V}_4]^{S_4^{(2)}}$ is well known
(see, for example, \cite{HST, NT}).

A minimal generating set for $\mathbb{K}[\mathcal{V}_4]^{S_4^{(2)}}$
can be chosen, for instance, as the following $9$ orbit averages:
\begin{gather*}
R(x_1)=\frac{1}{6}(x_1+x_2+x_3+x_4+x_5+x_6),\qquad
R(x_1^2)=\frac{1}{6}(x_1^2+\cdots+x_6^2),\\
R(x_1x_6)=\frac{1}{3}(x_1x_6+x_2x_5+x_3x_4),\qquad
R(x_1^3)=\frac{1}{6}(x_1^3+\cdots+x_6^3),\\
R(x_1x_2x_3)=\frac{1}{4}(x_1x_2x_3+x_1x_4x_5+x_2x_4x_6+x_3x_5x_6),\\
R(x_1^4)=\frac{1}{6}(x_1^4+\cdots+x_6^4),\qquad
R(x_1^5)=\frac{1}{6}(x_1^5+x_2^5+x_3^5+x_4^5+x_5^5+x_6^5),
\end{gather*}
as well as orbit sums of mixed monomials of degrees $3$ and $4$ of the type $R(x_1^2x_2)$ and
$R(x_1^3x_2)$; we omit their expanded formulas since they are cumbersome.
If one considers only simple graphs, then their invariant algebra
is generated by four invariants, see~\cite{Bbed},
$$
R(x_1),\qquad R(x_1x_6),\qquad R(x_1x_2),\qquad R(x_1x_2x_3).
$$

The generating system above describes the invariant algebra as an algebraic object; however, it readily yields a \emph{complete
invariant}
$$
\bigl(R(x_1),\,R(x_1x_6),\,R(x_1x_2),\,R(x_1x_2x_3)\bigr),
$$
whose values separate all
$11$ isomorphism types of simple graphs on $4$ vertices.

The invariant algebra $\mathbb{K}[\mathcal{V}_n]^{S_n^{(2)}}$ is known only for small $n$.
The algebra for $n=5$ has a minimal generating set consisting of $56$ polynomials of degree at most $9$, see \cite{NT}.
For $n\ge 6$ a complete description of the invariant algebra remains out of reach even for simple graphs:
the number of monomial orbits and the degrees of invariants increase sharply.
Therefore, constructing a complete graph invariant solely via polynomial invariants for large
$n$ is not feasible in practice in the general case.
In the following sections we propose a fundamentally different approach: we construct a complete graph invariant at the cost of passing to {semi-algebraic} functions, and we show that such a construction is essentially equivalent to the problem of {orbit canonization}.

\section{The Moving Frame Method for Finite Groups}

The modern concept of the moving frame method was formulated in \cite{Olver1998, Olver1999-1} as a universal mechanism for constructing invariants of smooth Lie group actions. In the classical differential--geometric context, the method is based on the existence of an equivariant map that sends an arbitrary point of the space to a заранее chosen orbit cross-section. This makes it possible, by solving normalization equations, to canonically eliminate the group parameters from the object. More recently, it was shown in \cite{Olver2023} that this geometric construction can be effectively adapted to finite and discrete groups. The key feature of a discrete moving frame is the replacement of smooth normalization by a combinatorial--algorithmic procedure---a deterministic choice of a unique orbit representative, which in effect identifies the construction of a moving frame with the canonization of the object. The fundamental invariants obtained in this way are piecewise-analytic (in particular, semi-algebraic) functions that form a complete system of invariants. Below we briefly present the theoretical foundations of this method for subsequent application to the graph isomorphism problem.

\subsection{Discrete moving frames}

Let a finite group $G$ act on a manifold $\mathcal M$:
$$
G\times \mathcal M\to \mathcal M,\qquad (g,x)\mapsto g\cdot x.
$$
A \textit{(right) moving frame} for this action is a map
$\rho:\mathcal M\to G$ satisfying the equivariance condition
\begin{equation}\label{equiv}
\rho(g\cdot x)=\rho(x)\,g^{-1},\qquad \forall\,g\in G,\ \forall\,x\in\mathcal M.
\end{equation}
This definition coincides with the definition in the smooth case \cite{Olver1999-1},
but now $\rho$ is no longer required
to be a smooth or continuous function.

As in the continuous case, a moving frame defines a \textit{normalization (canonization)} of a point on the
orbit:
\begin{equation*}
x^*=\rho(x)\cdot x\in \mathcal M.
\end{equation*}
It follows immediately from \eqref{equiv} that the canonized point depends
only on the orbit, that is,
for any $x\in\mathcal M$ and $g\in G$ we have
$$
(\,g\cdot x\,)^* \;=\; x^*.
$$
Indeed, by definition and equivariance \eqref{equiv} we obtain
$$
(\,g\cdot x\,)^*=\rho(g\cdot x)\cdot(g\cdot x)
=\bigl(\rho(x)g^{-1}\bigr)\cdot(g\cdot x)=\rho(x)\cdot x=x^*.
$$
Thus, in the finite case the moving frame realizes the same fundamental idea as in the continuous case:
all points of a single orbit are mapped to one and the same canonical representative $x^*$.

For the practical computation of invariants one introduces the operation of
\textit{invariantization} of functions.
Let $\mathcal{F}(\mathcal M)$ be some class of functions on $\mathcal M$. The action of $G$ on $\mathcal M$ induces an action on
functions by the rule
$$
(g\cdot F)(x)=F(g^{-1}\cdot x).
$$
Denote the set of $G$-invariants by
$$
\mathcal{F}(\mathcal M)^G=\{F\in\mathcal{F}(\mathcal M)\mid g\cdot F=F\ \ \forall g\in G\}.
$$

The \textit{invariantization} of a function $F\in\mathcal{F}(\mathcal M)$ with respect to a moving
frame $\rho$ is the function $\iota(F)$ defined by
\begin{equation*}
\iota(F)(x)=F\bigl(\rho(x)\cdot x\bigr)=F(x^*).
\end{equation*}
Invariantization is a projector onto invariants, $\iota:\mathcal{F}(\mathcal M)\to\mathcal{F}(\mathcal M)^G$.
In this sense, invariantization for finite groups is a conceptual analogue of the Reynolds operator, except that instead of discrete averaging over the group it performs orbit canonization.

In the smooth case, a cross-section $\mathcal K\subset\mathcal M$ is specified by normalization equations
for the Lie group parameters, and the moving frame is obtained by solving these equations
\cite{Olver1999-1}. For a finite group, the most natural choice is an {algorithmic}
cross-section defined by a rule that selects a canonical representative of each orbit.

Fix a total order $\prec$ on $\mathcal M$, for example,
the lexicographic order on $\mathbb{R}^m$ after fixing coordinates.
For each $x\in\mathcal M$ define the \textit{canonical orbit representative}
as the minimum over the orbit:
\begin{equation}\label{canon1}
\mathrm{can}(x)=\min\nolimits_{\prec}\{\,g\cdot x\mid g\in G\,\}.
\end{equation}
Then the cross-section can be formally written as
$$
\mathcal K=\{\,x\in\mathcal M\mid x=\mathrm{can}(x)\,\}.
$$
If the minimum in \eqref{canon1} is attained at a unique element of the
orbit (i.e., there is no ``tie-break''), then the moving frame is uniquely determined:
$$
\rho(x)=\text{the unique }g\in G\text{ such that }g\cdot x=\mathrm{can}(x).
$$
In the degenerate case when several minimal elements exist, one specifies a deterministic
tie-breaking rule. After that, $\rho$ becomes
totally defined, and the equivariance \eqref{equiv} holds
automatically, since the canonical representative does not change when moving along
the orbit.
\subsection{Completeness of the invariantized coordinates}

Let $\boldsymbol x=(x_1,\ldots,x_m)$ be local coordinates of a point $x\in\mathcal M$, where $m=\dim\mathcal M.$ Define the
invariantized coordinates by
$$
I_s(\boldsymbol x)=\iota(x_s)(\boldsymbol x)=x_s\bigl(\rho(x)\cdot x\bigr),\qquad s=1,\ldots,m,
$$
that is, the coordinates of the canonical representative $\boldsymbol x^*=\rho(\boldsymbol x)\cdot \boldsymbol x$.

Below we formulate a standard fact of the moving frame method: the invariantized
coordinates form a complete system of invariants in the corresponding class of functions.
In the smooth case, an analogous statement is contained, in particular, in
\cite[Theorem~4.5]{Olver1999-1}; in the finite situation the proof reduces to a purely group-theoretic verification.

\begin{theorem}\label{complete_invariants}
Let a finite group $G$ act on a manifold $\mathcal M$, and let
$\mathcal K\subset\mathcal M$ be a cross-section, while $\mathcal M_{\mathrm{reg}}\subset\mathcal M$ is a subset on which each $G$-orbit
intersects $\mathcal K$ {in exactly one point}. Let
$\rho:\mathcal M_{\mathrm{reg}}\to G$ be a discrete moving frame, i.e.,
$$
\rho(g\cdot \boldsymbol x)=\rho(\boldsymbol x)\,g^{-1},
\qquad g\in G,\ \boldsymbol x\in\mathcal M_{\mathrm{reg}}.
$$
Consider the map
$$
\Phi:\mathcal M_{\mathrm{reg}}\to\mathcal K,\qquad
\Phi(\boldsymbol x)=\rho(\boldsymbol x)\cdot \boldsymbol x.
$$
Then:
\begin{itemize}
\item[$(i)$]
If in local coordinates
$\Phi(\boldsymbol x)=(I_1(\boldsymbol x),\ldots,I_m(\boldsymbol x))$,
then $I_1,\ldots,I_m$ are $G$-invariants on $\mathcal M_{\mathrm{reg}}$ and form a
\emph{complete system of invariants} of the group $G$, i.e., they separate $G$-orbits on $\mathcal M_{\mathrm{reg}}$.

\item[$(ii)$]
For any $G$-invariant $F:\mathcal M_{\mathrm{reg}}\to \mathbb R$
there exists a function $\widetilde F:\Phi(\mathcal M_{\mathrm{reg}})\to \mathbb R$ such that
$$
F(\boldsymbol x)=\widetilde F\bigl(\Phi(\boldsymbol x)\bigr)
=\widetilde F\bigl(I_1(\boldsymbol x),\ldots,I_m(\boldsymbol x)\bigr),
\qquad \boldsymbol x\in\mathcal M_{\mathrm{reg}}.
$$
\end{itemize}
\end{theorem}

\begin{proof}
$(i)$ We first prove that $\Phi$ is a $G$-invariant map.
Let $g\in G$ and $\boldsymbol x\in\mathcal M_{\mathrm{reg}}$.
By the right equivariance of the moving frame we have
$$
\rho(g\cdot \boldsymbol x)=\rho(\boldsymbol x)\,g^{-1}.
$$
Then
$$
\Phi(g\cdot \boldsymbol x)
=\rho(g\cdot \boldsymbol x)\cdot (g\cdot \boldsymbol x)
=\bigl(\rho(\boldsymbol x)g^{-1}\bigr)\cdot (g\cdot \boldsymbol x)
=\rho(\boldsymbol x)\cdot \boldsymbol x
=\Phi(\boldsymbol x).
$$
Hence, $\Phi$ is constant on each $G$-orbit in $\mathcal M_{\mathrm{reg}}$.

Suppose that in local coordinates on $\mathcal K$ we have
$$
\Phi(\boldsymbol x)=(I_1(\boldsymbol x),\ldots,I_m(\boldsymbol x)).
$$
Since $\Phi(g\cdot \boldsymbol x)=\Phi(\boldsymbol x)$ for all $g\in G$, it follows that
$$
(I_1(g\cdot \boldsymbol x),\ldots,I_m(g\cdot \boldsymbol x))
=(I_1(\boldsymbol x),\ldots,I_m(\boldsymbol x)).
$$
Therefore, for each $k=1,\ldots,m$ we have
$$
I_k(g\cdot \boldsymbol x)=I_k(\boldsymbol x),
$$
so each $I_k$ is a $G$-invariant on $\mathcal M_{\mathrm{reg}}$.

We now prove completeness of the system $(I_1,\ldots,I_m)$.
We show the equivalence
$$
\boldsymbol y\in G\cdot \boldsymbol x
\quad\Longleftrightarrow\quad
(I_1(\boldsymbol y),\ldots,I_m(\boldsymbol y))=(I_1(\boldsymbol x),\ldots,I_m(\boldsymbol x)).
$$

\smallskip
{($\Rightarrow$)} If $\boldsymbol y=g\cdot \boldsymbol x$ for some $g\in G$,
then by $G$-invariance of $\Phi$ we have $\Phi(\boldsymbol y)=\Phi(\boldsymbol x)$,
and hence their coordinates on $\mathcal K$ coincide:
$$
(I_1(\boldsymbol y),\ldots,I_m(\boldsymbol y))=(I_1(\boldsymbol x),\ldots,I_m(\boldsymbol x)).
$$

\smallskip
{($\Leftarrow$)} Conversely, suppose that
$$
(I_1(\boldsymbol y),\ldots,I_m(\boldsymbol y))=(I_1(\boldsymbol x),\ldots,I_m(\boldsymbol x)).
$$
Then, by definition of the coordinate expression of $\Phi$,
$$
\Phi(\boldsymbol y)=\Phi(\boldsymbol x).
$$
Denote $u=\rho(\boldsymbol x)$ and $v=\rho(\boldsymbol y)$. Then
$$
u\cdot \boldsymbol x=\Phi(\boldsymbol x)=\Phi(\boldsymbol y)=v\cdot \boldsymbol y.
$$
Multiplying on the left by $v^{-1}$ yields
$$
\boldsymbol y=v^{-1}u\cdot \boldsymbol x.
$$
Since $v^{-1}u\in G$, it follows that $\boldsymbol y\in G\cdot \boldsymbol x$, i.e.,
$\boldsymbol x$ and $\boldsymbol y$ lie in the same $G$-orbit.
Thus, equality of the tuples $(I_1,\ldots,I_m)$ is equivalent to belonging to the same orbit,
and hence $(I_1,\ldots,I_m)$ forms a complete system of $G$-invariants on $\mathcal M_{\mathrm{reg}}$.
This proves part~(i).

\medskip
$(ii)$
Let $F:\mathcal M_{\mathrm{reg}}\to\mathbb R$ be a $G$-invariant, i.e.,
$$
F(g\cdot \boldsymbol x)=F(\boldsymbol x)
\qquad \forall\, g\in G,\ \forall\, \boldsymbol x\in\mathcal M_{\mathrm{reg}}.
$$
Since $\Phi(\boldsymbol x)=\rho(\boldsymbol x)\cdot \boldsymbol x$ lies on the same orbit as $\boldsymbol x$,
by invariance of $F$ we obtain
$$
F(\boldsymbol x)=F\bigl(\rho(\boldsymbol x)\cdot \boldsymbol x\bigr)=F\bigl(\Phi(\boldsymbol x)\bigr).
$$
Next we use the cross-section assumption: the set $\Phi(\mathcal M_{\mathrm{reg}})\subset \mathcal K$
contains exactly one representative from each $G$-orbit in $\mathcal M_{\mathrm{reg}}$.
Therefore one can define the function
$$
\widetilde F:\Phi(\mathcal M_{\mathrm{reg}})\to\mathbb R,
\qquad
\widetilde F(\boldsymbol z):=F(\boldsymbol z),\quad \boldsymbol z\in\Phi(\mathcal M_{\mathrm{reg}}).
$$
This is well-defined because for $\boldsymbol z\in\Phi(\mathcal M_{\mathrm{reg}})$ the value $F(\boldsymbol z)$
does not depend on the choice of $\boldsymbol x$ such that $\Phi(\boldsymbol x)=\boldsymbol z$:
any two such $\boldsymbol x$ lie in the same orbit (by part (i)), and $F$ is constant on orbits.

Hence, for any $\boldsymbol x\in\mathcal M_{\mathrm{reg}}$ we have
$$
F(\boldsymbol x)=F\bigl(\Phi(\boldsymbol x)\bigr)=\widetilde F\bigl(\Phi(\boldsymbol x)\bigr).
$$
If, in addition, $\Phi(\boldsymbol x)=(I_1(\boldsymbol x),\ldots,I_m(\boldsymbol x))$,
then we obtain the representation
$$
F(\boldsymbol x)=\widetilde F\bigl(I_1(\boldsymbol x),\ldots,I_m(\boldsymbol x)\bigr),
$$
which is what we needed to prove.
\end{proof}

\medskip
The proved theorem formalizes the key fact for our purposes:
{orbit canonization, and hence invariantization of the coordinates, yields a complete invariant};
however, unlike in algebraic invariant theory, these invariants are usually
not polynomials: they are defined via comparison and minimization over the orbit,
and therefore have an algorithmic, non-algebraic nature.

\subsection{Illustration of the method for the symmetric group $S_n$}

In this subsection we illustrate how, for the standard action of the symmetric group
$S_n$ on $\mathbb{R}^n$, the discrete moving frame method yields a complete
system of (non-algebraic) invariants.

Let $S_n$ act on $\mathbb{R}^n$ by permuting coordinates:
$$
\sigma\cdot (x_1,\ldots,x_n)=(x_{\sigma(1)},\ldots,x_{\sigma(n)}),
\qquad \sigma\in S_n.
$$
The orbit of a point $x\in\mathbb{R}^n$ consists of all permutations of its coordinates,
that is, it is the set of all vectors having the same multiset of values
$\{x_1,\ldots,x_n\}$.

Fix a total order on $\mathbb{R}^n$, for example the lexicographic
order. As a cross-section, in the finite sense, choose the chamber
$$
\mathcal K=\{\,y\in\mathbb{R}^n\mid y_1\le y_2\le\cdots\le y_n\,\}.
$$
This means: the canonical representative of an orbit is the vector whose coordinates
are arranged in nondecreasing order. In general position, when all coordinates are distinct,
we have strict inequalities $y_1<\cdots<y_n$, and the corresponding permutation is determined
uniquely.

For $x\in\mathbb{R}^n$ in general position define $\rho(x)\in S_n$ as
the unique permutation that sends $x$ into the cross-section:
$$
\rho(x)\cdot x\in \mathcal K.
$$
Equivariance,
$$
\rho(\sigma\cdot x)=\rho(x)\,\sigma^{-1},
$$
holds automatically, since the rule ``sort the coordinates'' depends
only on the orbit: permuting the input merely renames the coordinates, and $\rho$
compensates for this renaming, returning the result to the same chamber.

The normalized (canonical) point of the orbit is defined by
$$
x^*=\rho(x)\cdot x\in\mathcal K.
$$
In coordinates this is simply the ordered vector of values:
$$
x^*=(x_{(1)},x_{(2)},\ldots,x_{(n)}),
$$
where $x_{(1)}\le x_{(2)}\le\cdots\le x_{(n)}$ is an ordering of the multiset
$\{x_1,\ldots,x_n\}$. Thus, the canonical element of the orbit is its
{canonical representative in the ordering chamber}.

Invariantization of the coordinate functions $x_i$ is given by
$$
I_i(x)=\iota(x_i)(x)=x_i(\rho(x)\cdot x),\qquad i=1,\ldots,n.
$$
Since $\rho(x)\cdot x=x^*$, we obtain
$$
I_i(x)=x_{(i)},
$$
that is, $I_1$ is the minimum coordinate of the canonical representative, $I_n$ is its maximum, and the intermediate $I_k$ are
order statistics (second minimum, third minimum, etc.).
These functions are invariant with respect to $S_n$ by definition of invariantization.

The vector
$$
\Phi(x)=\bigl(I_1(x),\ldots,I_n(x)\bigr)=x^*,
$$
is a complete invariant of the orbit on the general position set and, after fixing
a tie-breaking rule, also in degenerate cases: two vectors $x,y$ lie in
the same $S_n$-orbit if and only if their canonical representatives coincide,
i.e., $\Phi(x)=\Phi(y)$. Equivalently, any $S_n$-invariant in the considered
class of functions is a function of the invariantized coordinates $I_1,\ldots,I_n$.

The canonization described above, which uses sorting in increasing order, is only the simplest choice.
Instead, one may use other rules for selecting a canonical orbit representative,
in particular: sorting in decreasing order, sorting by absolute value, by distance to the mean,
cluster criteria, or any other deterministic rule that depends
only on the orbit. Each such choice produces its own moving frame and its own
invariantized coordinates, but under correct canonization all of them yield
complete systems of invariants in the corresponding class of functions.

Note that the non-algebraic invariants, namely the ordered statistics $x_{(1)}\le \cdots \le x_{(n)}$ obtained by invariantizing the coordinates, are ``stronger'' than the classical algebraic invariants of $S_n$. Indeed, any symmetric polynomial depends only on the multiset of values $\{x_1,\dots,x_n\}$, and hence can be expressed as an ordinary polynomial in the ordered tuple $\bigl(x_{(1)},\dots,x_{(n)}\bigr)$, which is the canonical orbit representative. In particular,
$$
x_1+\cdots+x_n \;=\; x_{(1)}+\cdots+x_{(n)},\qquad
x_1x_2\cdots x_n \;=\; x_{(1)}x_{(2)}\cdots x_{(n)},
$$
and, more generally, the elementary symmetric polynomials can be written as
$$
e_k(x_1,\dots,x_n)\;=\!\!\sum_{1\le i_1<\cdots<i_k\le n}\!\! x_{i_1}\cdots x_{i_k}
\;=\!\!\sum_{1\le i_1<\cdots<i_k\le n}\!\! x_{(i_1)}\cdots x_{(i_k)}.
$$
Therefore, the classical algebra of symmetric polynomials arises as a subalgebra
generated by polynomial combinations of the non-algebraic invariants of the canonical form; in this sense, invariantization of the coordinates provides ``canonical
orbit coordinates'' from which symmetric polynomials are recovered automatically,
see~\cite{Olver2023}.

\section{A complete system of semi-algebraic invariants for $S_n^{(2)}$.}


In this section we construct a complete system of {semi-algebraic} invariants of the action of the
pair-symmetric group $S_n^{(2)}$ on the space of edge coordinates of
undirected weighted graphs. The main idea is to replace the classical
algebraic approach based on polynomial invariants by
orbit canonization for a finite group action. We show that lexicographic
canonization of the edge vector defines a discrete moving frame in the sense of Olver,
and that the invariants obtained by invariantizing the coordinates form a complete
graph invariant. This approach naturally connects invariant theory of
finite groups with the graph canonization problem and gives it a clear
invariant interpretation.

\subsection{Orbit canonization}

Recall that the action of the pair group $S_n^{(2)}$ is realized as a subgroup of coordinate permutations
in the space $\mathbb{R}^{\binom{n}{2}}$. For convenience, relabel the
edges $\{i,j\}$ ($i<j$) in lexicographic order
$$
(1,2)\prec(1,3)\prec\cdots\prec(1,n)\prec(2,3)\prec\cdots\prec(n-1,n),
$$
and introduce ordered variables with a single index
$$
x_1,x_2,\ldots,x_m,\qquad m=\binom{n}{2},
$$
where $x_s$ is the weight of the edge corresponding to the $s$-th pair $(i,j)$ in the above order.
Then a point of the edge-coordinate space, i.e., a weighted undirected graph,
is identified with a vector
$$
\boldsymbol{x}=(x_1,x_2,\ldots,x_m)\in\mathbb{R}^m.
$$

Each vertex permutation $\sigma\in S_n$ induces a permutation of edges, and hence a permutation of edge coordinates; thus there exists an element
$\tau\in S_n^{(2)}\subset S_m$ which permutes the indices $\{1,\dots,m\}$ in accordance with the
relabeling of edges. That is, for each $\tau\in S_n^{(2)}$ we have
$$
(\tau\cdot \boldsymbol{x})_s = x_{\tau^{-1}(s)},\qquad s=1,\ldots,m.
$$
Equivalently, if $s$ corresponds to the edge $\{i,j\}$, then $\tau(s)$ corresponds to the edge
obtained from $\{i,j\}$ under the action of some vertex permutation, after which the pair
is reordered according to the rule $i<j$.

The orbit
$$
O_{\boldsymbol{x}}=\{\tau\cdot \boldsymbol{x}\mid \tau\in S_n^{(2)}\}
$$
consists of all representations of one and the same weighted undirected graph under
different vertex labelings. Hence, invariants of this action are precisely graph invariants
with respect to isomorphism.

\medskip

To construct a discrete moving frame, fix the lexicographic order on
$\mathbb{R}^m$ and set
$$
V(\boldsymbol{x})=(x_1,x_2,\ldots,x_m)\in\mathbb{R}^m.
$$
Then the \emph{canonical orbit representative} is defined by
\begin{equation}\label{can_sn2}
\mathrm{can}(\boldsymbol{x})\;=\;\min\nolimits_{\text{lex}}
\bigl\{\,V(\tau\cdot \boldsymbol{x})\ \bigm|\ \tau\in S_n^{(2)}\,\bigr\},
\end{equation}
that is, the element of the orbit whose edge vector is lexicographically minimal.
In general position, the minimum is attained at a unique orbit element; in
degenerate situations, when the minimum is attained at several elements, one fixes a
deterministic tie-breaking rule that makes the choice unique.

As a result, we obtain the standard definition of graph canonization: we choose a
canonical representation of edge weights
that brings the edge vector to a minimal form. Hence, graph canonization acquires the precise meaning of a
{canonical cross-section of the orbits of the action of $S_n^{(2)}$}.

In the finite case it is convenient to think of the cross-section as the set of all already canonical
objects:
$$
\mathcal{K}=\{\,\boldsymbol{x}\in\mathbb{R}^m\mid 
V(\boldsymbol{x})=\mathrm{can}(\boldsymbol{x})\,\}.
$$
That is, $\mathcal{K}$ contains exactly one representative from each orbit, for a
fixed tie-breaking rule.

\medskip

Define $\rho(\boldsymbol{x})\in S_n^{(2)}$ as the uniquely chosen element
realizing the minimum in \eqref{can_sn2}:
$$
\rho(\boldsymbol{x})=\arg\min_{\tau\in S_n^{(2)}} V(\tau\cdot \boldsymbol{x}),
\qquad\text{that is}\qquad
\rho(\boldsymbol{x})\cdot \boldsymbol{x}\in\mathcal{K}
\ \ \text{and}\ \
V(\rho(\boldsymbol{x})\cdot \boldsymbol{x})=\mathrm{can}(\boldsymbol{x}).
$$
Then $\rho$ is a right discrete moving frame: for any
$\tau\in S_n^{(2)}$ we have
$$
\rho(\tau\cdot \boldsymbol{x})=\rho(\boldsymbol{x})\,\tau^{-1},
$$
since the orbit set is the same and the canonical orbit representative does not depend on the
choice of its element.

\subsection{Invariantization and a complete system of invariants.}
Invariantization of a function is defined by
$$
\iota(F)(\boldsymbol{x}) = F\left(\rho(\boldsymbol{x}) \cdot \boldsymbol{x} \right).
$$
Apply it to the coordinate functions $x_s$:
$$
I_s(\boldsymbol{x}) = \iota(x_s)(\boldsymbol{x}) = \left(\rho(\boldsymbol{x}) \cdot \boldsymbol{x}\right)_s, \qquad s = 1, \ldots, m.
$$
Hence, the vector of invariantized coordinates
$$
\boldsymbol{I}(\boldsymbol{x}) = \left(I_1(\boldsymbol{x}), I_2(\boldsymbol{x}), \ldots, I_m(\boldsymbol{x}) \right) = \rho(\boldsymbol{x}) \cdot \boldsymbol{x},
$$
coincides with the canonical orbit representative, i.e.,
$$
\boldsymbol{I}(\boldsymbol{x}) = \operatorname{can}(\boldsymbol{x}).
$$

According to Theorem~\ref{complete_invariants}, this set of invariants is complete: for any
$\boldsymbol{x}, \boldsymbol{y} \in \mathbb{R}^m$ we have
$$
\boldsymbol{I}(\boldsymbol{x}) = \boldsymbol{I}(\boldsymbol{y}) \quad \Longleftrightarrow \quad \exists \, \tau \in S_n^{(2)} : \ \boldsymbol{y} = \tau \cdot \boldsymbol{x}.
$$

Equivalently, for any $S_n^{(2)}$-invariant $F$ in the fixed class of functions there exists a function $\widetilde{F}$ such that
$$
F(\boldsymbol{x}) = \widetilde{F}\left(I_1(\boldsymbol{x}), I_2(\boldsymbol{x}), \ldots, I_m(\boldsymbol{x}) \right),
$$
since $F$ is constant on the orbit and the invariantized coordinates are the coordinates of the canonical representative of this orbit.

The most important practical consequence is that for an {arbitrary} orbit representative, i.e., for an arbitrary vertex labeling of the graph,
invariantization of the coordinates returns one and the same canonical object:
$$
\mathrm{can}(\boldsymbol{x})=\rho(\boldsymbol{x})\cdot \boldsymbol{x}=\bigl(I_{1}(\boldsymbol{x}),I_{2}(\boldsymbol{x}),\ldots,I_{m}(\boldsymbol{x})\bigr),
$$
and hence invariantization realizes canonization.

\subsection{Semi-algebraic nature of the obtained invariants}

Since the constructed invariants $I_k(\boldsymbol{x})$ are not polynomials, it is important
to clarify their mathematical status. As emphasized by Olver~\cite{Olver2023},
invariants of discrete groups naturally go beyond the classical algebras
of polynomial invariants and belong to broader classes of functions. In our construction,
a natural such class is that of {semi-algebraic} functions.

A subset $A\subset\mathbb{R}^n$ is called {semi-algebraic} if it can be represented
as a finite Boolean combination of sets defined by polynomial
equalities and inequalities.
A function $f\colon A\to\mathbb{R}$ is called {semi-algebraic} if its graph
$$
\Gamma(f)=\{(x,y)\in A\times\mathbb{R}\mid y=f(x)\}\subset\mathbb{R}^{n+1},
$$
is a semi-algebraic set; see \cite{BC}.

Typical examples of semi-algebraic functions include polynomials and rational functions
(on their domains of definition), the absolute value function $|x|=\sqrt{x^2}$, and the functions
$\min(x,y)$ and $\max(x,y)$.

\begin{theorem}\label{semi_alg_I}
The invariant functions $I_k(\boldsymbol{x})$, defined as the $k$-th coordinates of the
canonical representative $\mathrm{can}(\boldsymbol{x})$ of the orbit of the action of $S_n^{(2)}$, defined by the lexicographic minimum with a fixed tie-breaking rule,
are semi-algebraic functions.
\end{theorem}

\begin{proof}
By definition of canonization, for each $\boldsymbol{x}\in\mathbb{R}^m$ we fix
a deterministic tie-breaking rule and define
$$
\mathrm{can}(\boldsymbol{x})=\min\nolimits_{\text{lex}}\{\,\tau\cdot \boldsymbol{x}\mid \tau\in S_n^{(2)}\,\},
\qquad
\rho(\boldsymbol{x})=\arg\min_{\tau\in S_n^{(2)}} (\tau\cdot \boldsymbol{x}),
$$
with respect to the lexicographic order on $\mathbb{R}^m$. Then $I_k(\boldsymbol{x})$ is the $k$-th
coordinate of the vector $\mathrm{can}(\boldsymbol{x})$, i.e.,
$$
I_k(\boldsymbol{x})=\bigl(\mathrm{can}(\boldsymbol{x})\bigr)_k.
$$

Consider the graph of the function $I_k$:
$$
\Gamma(I_k)=\{(\boldsymbol{x},y)\in\mathbb{R}^m\times\mathbb{R}\mid y=I_k(\boldsymbol{x})\}.
$$
By the definition of $\mathrm{can}(\boldsymbol{x})$ we have the equivalence
\begin{equation}\label{graph_Ik}
y=I_k(\boldsymbol{x})
\iff
\bigvee_{\tau\in S_n^{(2)}}
\Bigl(
y=(\tau\cdot\boldsymbol{x})_k
\ \wedge\
\bigwedge_{\sigma\in S_n^{(2)}}\bigl(\tau\cdot\boldsymbol{x}\preceq \sigma\cdot\boldsymbol{x}\bigr)
\Bigr),
\end{equation}
where $\preceq$ denotes the lexicographic order on $\mathbb{R}^m$.
Since $S_n^{(2)}$ is finite, we obtain a finite conjunction, and hence the description of $\Gamma(I_k)$ reduces to a finite Boolean combination of linear equalities and inequalities.

Next, the condition $\tau\cdot\boldsymbol{x}\preceq \sigma\cdot\boldsymbol{x}$ expands
into a finite Boolean combination of equalities and strict inequalities between coordinates:
$$
(\tau\cdot\boldsymbol{x})_1<(\sigma\cdot\boldsymbol{x})_1
\ \ \vee\ \
\Bigl((\tau\cdot\boldsymbol{x})_1=(\sigma\cdot\boldsymbol{x})_1\ \wedge\
(\tau\cdot\boldsymbol{x})_2<(\sigma\cdot\boldsymbol{x})_2\Bigr)
\ \ \vee\ \ \cdots,
$$
up to the $m$-th step. Each such equality or inequality is a polynomial (linear)
condition on the variables $\boldsymbol{x}$.

Therefore, condition \eqref{graph_Ik} describes $\Gamma(I_k)$ as a finite Boolean combination
of sets defined by linear equations and inequalities in the variables $(\boldsymbol{x},y)$.
By the definition of semi-algebraic sets, this implies that $\Gamma(I_k)$ is semi-algebraic,
and hence $I_k$ is a semi-algebraic function.
\end{proof}

The obtained result formalizes graph canonization as the problem of constructing
semi-algebraic invariants: a complete invariant is given not by polynomials, as in
classical invariant theory, but by an algorithmic rule selecting an orbit representative of a
finite group. This translates the ``canonical labeling'' problem into the
language of moving frames: $\rho(\boldsymbol{x})$ is a discrete moving frame, and
$\mathrm{can}(\boldsymbol{x})=\rho(\boldsymbol{x})\cdot \boldsymbol{x}$ is a normalization, which is an orbit invariant.
Thus, canonization appears as a natural analogue of invariantization
for finite groups, and the invariantized edge coordinates, according to Theorem~\ref{complete_invariants}, form a complete
system of invariants in the chosen class of functions.

It is worth noting that the moving frame $\rho(\boldsymbol x)$ is uniquely defined only
for graphs with a trivial automorphism group.
If a graph has internal symmetries, then the minimum in the lexicographic
canonization problem may be attained on an entire subset of permutations.
To ensure determinism, one fixes a \emph{tie-breaking} rule:
for example, among all permutations that send $\boldsymbol x$ to the
lexicographically minimal element of its orbit, one chooses the smallest permutation with respect to a prescribed order.
Then $\rho(\boldsymbol x)$ is uniquely defined, and the normalized vector
$$
\iota(\boldsymbol x)=\rho(\boldsymbol x)\cdot \boldsymbol x
$$
is unique and does not depend on the ambiguity of the choice within that subset of permutations.

\subsection{Example for $n=4$}

For $n=4$ we have $m=\binom{4}{2}=6$ edge coordinates of the graph. Fix the edge numbering
in lexicographic order
$$
(12)\prec(13)\prec(14)\prec(23)\prec(24)\prec(34),
$$
and introduce the variables
$$
(x_1,x_2,x_3,x_4,x_5,x_6)=(x_{12},x_{13},x_{14},x_{23},x_{24},x_{34}).
$$
Then we identify a point of the edge-coordinate space (a weighted graph) with the vector
$$
\boldsymbol{x}=(x_1,x_2,x_3,x_4,x_5,x_6)\in\mathbb{R}^6.
$$
The group $S_4^{(2)}\cong S_4$ has $24$ elements, so the orbit $O_{\boldsymbol{x}}$
typically contains $24$ distinct six-dimensional vectors.

We compare vectors lexicographically and define the moving frame as the coordinate permutation
that sends $\boldsymbol{x}$ to the lexicographically minimal element of
its orbit:
$$
\rho(\boldsymbol{x})=\arg\min_{\tau\in S_4^{(2)}} V(\tau\cdot \boldsymbol{x})
=\arg\min_{\tau\in S_4^{(2)}} (\,(\tau\cdot \boldsymbol{x})_1,\ldots,(\tau\cdot \boldsymbol{x})_6\,).
$$
The normalized point
$$
\boldsymbol{x}^*=\rho(\boldsymbol{x})\cdot \boldsymbol{x}
$$
is a canonical representation of the graph (a canonical orbit representative) in the chosen edge order.

A complete system of invariants in this case consists of the six-tuple
$$
I_1(\boldsymbol{x}),\ I_2(\boldsymbol{x}),\ \ldots,\ I_6(\boldsymbol{x}),
$$
where
$$
(I_1(\boldsymbol{x}),\ldots,I_6(\boldsymbol{x}))
=
\bigl(\rho(\boldsymbol{x})\cdot \boldsymbol{x}\bigr)
=
\boldsymbol{x}^*.
$$
By definition of $\rho(\boldsymbol{x})$, this six-dimensional vector is lexicographically
minimal among all $24$ elements of the orbit $O_{\boldsymbol{x}}$, i.e.,
$$
\boldsymbol{x}^* = \mathrm{can}(\boldsymbol{x}).
$$

\medskip
In other words, $I_1(\boldsymbol{x})$ is the first coordinate of the normalized
vector $\boldsymbol{x}^*$, $I_2(\boldsymbol{x})$ is the second one, and so on up to
$I_6(\boldsymbol{x})$, which together completely reconstruct the canonical orbit
representative. By Theorem~\ref{semi_alg_I}, each invariant $I_s(\boldsymbol{x})$ is a {semi-algebraic}
function of the initial coordinates $\boldsymbol{x}$.

\medskip
Note that the choice of canonization is not unique: instead of the lexicographic minimum
one may use other deterministic rules. All such variants yield different moving frames, but, under
the assumption that the canonical representative is chosen uniquely,
each of them produces a complete system of invariants via invariantization of the coordinates.

Although our approach uses exhaustive search for minimization, modern canonical labeling algorithms such as \texttt{nauty} or \texttt{Traces}~\cite{McKayPiperno2014} in fact compute the same discrete moving frame much more efficiently. Instead of global minimization, they use an {individualization--refinement} strategy that progressively reduces the group of acting symmetries. In the terminology of the present work, this can be interpreted as a successive approximation to the canonical orbit cross-section.

\section{Conclusions}

In this work we propose a new interpretation of the graph canonization problem as the problem of constructing {semi-algebraic invariants} of the action of the pair group $S_n^{(2)}$ on the space of edge coordinates. In contrast to the classical approach in algebraic invariant theory, where the main effort is aimed at describing the ring of polynomial invariants, in our approach the central object is the {canonization algorithm}, formalized as a discrete {moving frame} in the sense of Olver.

This viewpoint allows us to connect two previously separated worlds: the purely algorithmic procedure of finding a canonical form and the differential-geometric theory of orbit normalization. We have shown that canonical labeling of a graph is nothing but the choice of a cross-section in the orbit space, and canonization itself is the process of invariantizing coordinates via a moving frame. This elevates the practical canonization procedure to the status of a rigorous mathematical object amenable to classification and structural analysis.

In further research, special attention is planned to be devoted to a systematic description of different canonization rules as distinct classes of discrete frames, which will make it possible to carry out a comparative analysis of their stability and discriminative power with respect to degenerate structures. This approach opens prospects for a deeper study of the relationship between the proposed method and computational complexity theory, where the complexity of graph canonization can be interpreted through the structural complexity of constructing the corresponding moving frame and through the analysis of the geometry of a fundamental domain in the space of edge coordinates.

The main outcome of the work is the establishment of a conceptual framework: {graph canonization is a discrete moving frame, and the coordinates of the canonical representative form a complete system of semi-algebraic invariants of the action of $S_n^{(2)}$}. This creates a bridge between the invariant theory of Lie group actions, algebraic combinatorics, and algorithmic graph theory, proposing a new language for addressing fundamental problems of isomorphism and structural recognition.

\end{document}